\documentclass[reqno]{amsart}
\usepackage{amssymb}

\usepackage[colorlinks=true]{hyperref}
\hypersetup{urlcolor=blue, citecolor=red}

\newtheorem{Theorem}{Theorem}[section]

\newtheorem{Lemma}[Theorem]{Lemma}
\newtheorem{Proposition}[Theorem]{Proposition}

\numberwithin{equation}{section}

\usepackage{tocvsec2}

\def\C {\mathbb C}
\def\R {\mathbb R}

\newcommand{\Spec}{\operatorname{Spec}}
\newcommand{\<}{\langle}
\renewcommand{\>}{\rangle}
\newcommand{\id}{\operatorname{Id}}
\newcommand{\inclusion}{\hookrightarrow}

\newcommand{\p}{\partial}
\newcommand{\Vol}{\operatorname{Vol}}
\renewcommand{\Re}{\operatorname{Re}}

\begin{document}
\title[A note on linear time-harmonic Maxwell equations]{A note on time-harmonic Maxwell equations on Riemannian manifolds}

\author[Yernat M. Assylbekov]{Yernat M. Assylbekov}
\address{Department of Computational and Applied Mathematics, Rice University, Houston, TX 77005, USA}
\email{yernat.assylbekov@gmail.com}

\maketitle

\begin{abstract}
In this paper we consider boundary value problems in electromagnetism. We prove well-posedness results for the time-harmonic Maxwell equations in the setting of Riemannian manifolds. We also consider the eigenvalue problem the homogeneous time-harmonic Maxwell equations with zero boundary conditions.
\end{abstract}


\section{Introduction}\label{section::introduction}

\settocdepth{section}

In the current note, which serves as the author's personal reference, we consider boundary value problems in electromagnetism.  We prove well-posedness results for the time-harmonic Maxwell equations in the setting of Riemannian manifolds. Some of these results are well known for bounded domains in $\R^3$; see classical references \cite{KirschHettlich2015,Monk2003}. To the best of author's knowledge, there are few related literatures on Riemannian geometries \cite{KenigSaloUhlmann2011a,MitreaMitrea2002,Mitrea2004,Mitrea2001}. However, these results assume that either electromagnetic parameters being constantly one or too regular.\smallskip

Let $(M,g)$ be a compact $3$-dimensional Riemannian manifold with smooth boundary. By $d$ and $*$ we denote the exterior derivative and the Hodge star operator on $(M,g)$, respectively. Consider the time-harmonic Maxwell equations for complex $1$-forms $E$ and $H$
\begin{equation}\label{eqn::Maxwell homogenous in appendix}
\begin{cases}
* dE=i\omega\mu H,\\
* dH=-i\omega \varepsilon E,
\end{cases}
\end{equation}
where $\omega>0$ is a fixed frequency. The complex functions $\mu$ and $\varepsilon$ represent the material parameters (permettivity and permeability, respectively). We assume that $\varepsilon,\mu$ are in $ L^\infty(M)$ and satisfy
\begin{equation}\label{eqn::uniform ellipticity}
\Re\varepsilon, \Re\mu\ge c
\end{equation}
for some constant $c>0$.\smallskip

Let $\imath:\p M\inclusion M$ be the canonical inclusion. Then we introduce tangential trace of $m$-forms by
$$
\mathbf{t}:C^\infty\Omega^{m}(M)\to C^\infty\Omega^{m}(\p M),\quad \mathbf{t}(w)=\imath^*(w),\quad w\in C^\infty\Omega^{m}(M).
$$
We work with the following Hilbert space which is the largest domain of $d$ acting on $m$-forms:
$$
H_d\Omega^{m}(M):=\{w\in L^2\Omega^{m}(M):dw\in L^2\Omega^{m+1}( M)\}
$$
endowed with the inner product
$$
(w_1|w_2)_{H_{d}\Omega^m(M)}:=(w_1|w_2)_{L^2\Omega^{m}(M)}+(dw_1|dw_2)_{L^2\Omega^{m+1}(M)}
$$
and the corresponding norm $\|w\|_{H_{d}\Omega^{m}(M)}^2:=(w|w)_{H_{d}\Omega^m(M)}$. Then the tangential trace operator has its extensions to bounded operators $\mathbf{t}:H_d\Omega^{m}(M)\to H^{-1/2}\Omega^{m}(\p M)$ and $\mathbf{t}:H^1\Omega^{m}(M)\to H^{1/2}\Omega^{m}(\p M)$. In fact, $\mathbf{t}$ is bounded from $H_d\Omega^{m}(M)$ into
$$
TH_d\Omega^{m}(\p M):=\{\mathbf{t}(w):w\in H_d\Omega^{m}(M)\}
$$
with the topology defined by the norm
$$
\|f\|_{TH_d\Omega^{m}(\p M)}:=\inf\{\|w\|_{H_d\Omega^{m}(M)}:\mathbf{t}(w)=f,\,w\in H_d\Omega^{m}(M)\}.
$$
We refer the reader to Section~\ref{section::on H_d and H_delta spaces} for more details. Let us now state our main results.\smallskip

For a given $f\in TH_d\Omega^1(\p M)$, we consider the time-harmonic Maxwell equations \eqref{eqn::Maxwell homogenous in appendix} with the tangential boundary condition $\mathbf{t}(E)=f$, where $\omega\in\C$ is fixed.\smallskip

The following theorem is the first main result of this paper.

\begin{Theorem}\label{thm::well posedness new version homogeneous}
Let $(M,g)$ be a compact $3$-dimensional Riemannian manifold with smooth boundary. Suppose $\varepsilon,\mu\in L^\infty(M)$ satisfy \eqref{eqn::uniform ellipticity}. There is a discrete subset $\Sigma$ of $\C$ such that for all $\omega\notin \Sigma$ and for a given $f\in TH_d\Omega^1(\p M)$ the Maxwell equation \eqref{eqn::Maxwell homogenous in appendix} with $\mathbf t(E)=f$ has a unique solution $(E,H)\in H_{d}\Omega^1(M)\times H_{d}\Omega^1(M)$ satisfying
$$
\|E\|_{H_{d}\Omega^1(M)}+\|H\|_{H_{d}\Omega^1(M)}\le C\|f\|_{TH_d\Omega^1(\p M)}
$$
for some constant $C>0$ independent of $f$.
\end{Theorem}

To prove Theorem~\ref{thm::well posedness new version homogeneous}, we consider the following non-homogeneous problem. Let $J_e$ and $J_m$ be $1$-forms on $M$ representing current sources. We consider the non-homogenous time-harmonic Maxwell equations
\begin{equation}\label{eqn::Maxwell in appendix}
\begin{cases}
* dE=i\omega\mu H+J_m,\\
* dH=-i\omega \varepsilon E+J_e
\end{cases}
\end{equation}
We also work with the space of differential forms in $H_{d}\Omega^1(M)$ with zero tangential traces
$$
H_{d,0}\Omega^1(M):=\{w\in H_{d}\Omega^1(M):\mathbf{t}(w)=0\}.
$$
Our second main result is as follows.

\begin{Theorem}\label{thm::well posedness new version}
Let $(M,g)$ be a compact $3$-dimensional Riemannian manifold with smooth boundary. Suppose $\varepsilon,\mu\in L^\infty(M)$ satisfy \eqref{eqn::uniform ellipticity} and $J_e,J_m\in L^2\Omega^1(M)$. There is a discrete subset $\Sigma$ of $\C$ such that for all $\omega\notin \Sigma$ the Maxwell's system \eqref{eqn::Maxwell in appendix} has a unique solution $(E,H)\in H_{d,0}\Omega^1(M)\times H_{d}\Omega^1(M)$ satisfying
$$
\|E\|_{H_{d}\Omega^1(M)}+\|H\|_{H_{d}\Omega^1(M)}\le C(\|J_e\|_{L^2\Omega^1(M)}+\|J_m\|_{L^2\Omega^1(M)})
$$
for some constant $C>0$ independent of $J_e$ and $J_m$.
\end{Theorem}


Finally, we also consider the eigenvalue problem for the boundary value problem
\begin{equation}\label{eqn::Maxwell homogeneous}
\begin{cases}
*dE=i\omega\mu H,\\
*dH=-i\omega \varepsilon E
\end{cases}
\quad\text{and}\quad \mathbf t(E)=0
\end{equation}
under the additional assumption that both $\varepsilon$ and $\mu$ are real-valued.

\begin{Theorem}\label{thm::spectral theorem for Maxwell equation homogeneous}
Let $(M,g)$ be a compact $3$-dimensional Riemannian manifold with smooth boundary and let $\varepsilon,\mu\in L^\infty(M)$ be real-valued and satisfying \eqref{eqn::uniform ellipticity}. There is a sequence of positive numbers $\{\omega_k\}_{k=1}^\infty$ and the corresponding sequence
$$
\{(e_k,h_k)\}_{k=1}^\infty\in H_{d,0}\Omega^1(M)_{0,\varepsilon}\times H_{d}\Omega^1(M)_\mu
$$
satisfying
\begin{equation}\label{eqn::Maxwell homogenous for eigenvalues and eigenfunctions}
\begin{cases}
* de_k=i\omega_k\mu h_k,\\
* dh_k=-i\omega_k \varepsilon e_k.
\end{cases}
\end{equation}
The eigenvalues $\omega_k>0$ have finite multiplicity, $0<\omega_1\le\omega_2\le \dots\to\infty$ as $k\to\infty$. The set $\{e_k\}_{k=1}^\infty$ forms an orthonormal basis in $H_{d,0}\Omega^1(M)_{0,\varepsilon}$ with respect to the inner product $(\cdot,\cdot)_{L^2_\varepsilon\Omega^1(M)}:=(\varepsilon\cdot|\cdot)_{L^2\Omega^1(M)}$ and the set $\{h_k\}_{k=1}^\infty$ forms a basis in $H_d\Omega^1(M)_\mu$ which is orthonormal with respect to the inner products $(\cdot,\cdot)_{L^2_\mu\Omega^1(M)}:=(\mu\cdot|\cdot)_{L^2\Omega^1(M)}$. Moreover, $\omega=0$ is an eigenvalue as well with infinite dimensional eigenspace $H_{d,0}(0,\Omega^1(M))\times H_{d}(0,\Omega^1(M))$, where $H_{d,0}(0,\Omega^1(M)):=H_{d,0}\Omega^1(M)\cap H_{d}(0,\Omega^1(M))$.
\end{Theorem}

These results were obtained for bounded domains in the Euclidean space $\R^3$ assuming $J_m=0$; see \cite{KirschHettlich2015,Monk2003} and references therein. On Riemannian manifolds, Theorem~\ref{thm::well posedness new version homogeneous} was proven in \cite{KenigSaloUhlmann2011a} under the assumption that $\varepsilon,\mu\in C^k(M)$, $k\ge 2$.\smallskip

We would also like to remark that our results are stated for the case when $\varepsilon$ and $\mu$ are independent of the frequency $\omega$. However, there are many applications where the electromagnetic parameters depend on $\omega$. For instance, in lossy materials, $\varepsilon=\varepsilon_0+i\sigma/\omega$ and $\mu>0$ where $\varepsilon_0>0$ and $\sigma\ge 0$. The methods used in the present paper can be extended to this particular case as in \cite{KirschHettlich2015,Monk2003}.

\subsection*{Structure of the paper} The paper is organized as follows. In Section~\ref{section::preliminaries} we briefly present basic facts on differential forms and trace operators. Then in Section~\ref{section::on H_d and H_delta spaces}  we show that the trace operators can be extended to $H_d\Omega^m(M)$ and to the closely related space $H_\delta\Omega^m(M)$. In that section we study some other important properties of those spaces. Next, in Section~\ref{section::well-posedness} we prove appropriated  Helmholtz decompositions for $1$-forms and certain compact embedding results. Sections~\ref{section::proof of thm2}, \ref{section::proof of thm1} and \ref{section::spectral problem} are devoted to proofs of main results.


\section{Preliminaries}\label{section::preliminaries}

In this section we briefly present basic facts on differential forms and trace operators. For more detailed exposition we refer the reader to the manuscript of Schwarz \cite{Schwarz1995}.\smallskip

Let $(M,g)$ be a compact oriented $n$-dimensional Riemannian manifold with smooth boundary. The inner product of tangent vectors with respect to the metric $g$ is denoted by $\<\cdot,\cdot\>_g$, and $|\cdot|_g$ is the notation for the corresponding norm. By $|g|$ we denote the determinant of $g=(g_{ij})$ and $(g^{ij})$ is the inverse matrix of $(g_{ij})$. Finally, there is the induced metric $\imath^*g$ on $\p M$ which gives a rise to the inner product $\<\cdot,\cdot\>_{\imath^*g}$ of vectors tangent to $\p M$.

\subsection{Basic notations for differential forms}\label{sec::2.1}
In what follows, for $F$ some function space ($C^k$, $L^p$, $H^k$, etc.), we denote by $F\Omega^m(M)$ the corresponding space of $m$-forms. In particular, the space of smooth $m$-forms is denoted by $C^\infty\Omega^{m}(M)$. Let $*:C^\infty\Omega^{m}(M)\to C^\infty\Omega^{n-m}(M)$ be the Hodge star operator. For real valued $\eta,\zeta\in C^\infty\Omega^{m}(M)$, the inner product with respect to $g$ is defined in local coordinates as
$$
\<\eta,\zeta\>_g=*(\eta\wedge*\zeta)=g^{i_1 j_1}\cdots g^{i_m j_m}\eta_{i_1\dots i_m}\zeta_{j_1\dots j_m}.
$$
This can be extended as a bilinear form on complex valued forms. We also write $|\eta|_g^2=\<\eta,\overline\eta\>_g$. The inner product on $L^2\Omega^{m}(M)$ is defined as
$$
(\eta|\zeta)_{L^2\Omega^{m}(M)}=\int_M\<\eta,\overline{\zeta}\>_g\,d\Vol_g=\int_M\eta\wedge*\overline{\zeta},\qquad \eta,\zeta\in L^2\Omega^m(M),
$$
where $d\Vol_g=* 1=|g|^{1/2}\,dx^1\wedge\cdots\wedge\,dx^n$ is the volume form. The corresponding norm is $\|\cdot\|_{L^2\Omega^{m}(M)}^2=(\cdot|\cdot)_{L^2\Omega^{m}(M)}$. Using the definition of the Hodge star operator $*$, it is not difficult to check that
\begin{equation}\label{eqn::Hodge star is isometry in L^2}
(\eta|\zeta)_{L^2\Omega^{m}(M)}=(*\eta|*\zeta)_{L^2\Omega^{n-m}(M)}.
\end{equation}
Let $d:C^\infty\Omega^{m}(M)\to C^\infty\Omega^{m+1}(M)$ be the external differential. Then the codifferential $\delta:C^\infty\Omega^{m}(M)\to C^\infty\Omega^{m-1}(M)$ is defined as
$$
(d\eta|\zeta)_{L^2\Omega^{m}(M)}=(\eta|\delta\zeta)_{L^2\Omega^{m-1}(M)}
$$
for all  $\eta\in C^\infty_0\Omega^{m-1}(M^{\rm int})$, $\zeta\in C^\infty\Omega^{m}(M)$. The Hodge star operator $*$ and the codifferential $\delta$ have the following properties when acting on $C^\infty\Omega^{m}(M)$:
\begin{equation}\label{eqn::delta in terms of * and d}
*^2=(-1)^{m(n-m)},\quad \delta=(-1)^{m(n-m)-n+m-1}*(d*\cdot).
\end{equation}

For a given $\xi\in C^\infty\Omega^1(M)$, the interior product $i_\xi:C^\infty\Omega^{m}(M)\to C^\infty\Omega^{m-1}(M)$ is the contraction of differential forms by $\xi$. In local coordinates,
$$
i_\xi \eta=g^{ij}\xi_i\,\eta_{ji_1\dots i_{m-1}},\quad \eta\in C^\infty\Omega^{m}(M).
$$
The interior product acts on exterior products in the following way
\begin{equation}\label{eqn::contraction on wedge}
i_\xi(\eta\wedge\zeta)=i_\xi\eta\wedge \zeta+(-1)^m \eta\wedge i_\xi\zeta,\quad\eta\in C^\infty\Omega^{m}(M),\,\zeta\in C^\infty\Omega^{k}(M).
\end{equation}
It is the formal adjoint of $\xi$, in the inner product $\<\cdot,\cdot\>_g$ on real valued forms, and has the following expression
\begin{equation}\label{eqn::expression for i_X}
i_\xi\eta=(-1)^{n(m-1)}*(\xi\wedge*\eta),\quad\eta\in C^\infty\Omega^{m}(M).
\end{equation}
Using this, one can also show that
\begin{equation}\label{eqn::divergence of fw}
\delta(fw)=f\delta w-i_{df}w,\quad f\in C^\infty(M),\quad w\in C^\infty\Omega^m(M).
\end{equation}

The Hodge Laplacian acting on $\Omega^{m}(M)$ is defined by $-\Delta=d\delta+\delta d$.\smallskip

Finally, the inner product on $L^2\Omega^m(\p M)$ is given by
$$
(u|v)_{L^2\Omega^{m}(\p M)}=\int_{\p M}\<u,\overline v\>_{\imath^* g}\,d\sigma_{\p M},\qquad u,v\in L^2\Omega^m(\p M),
$$
where $\<\cdot,\cdot\>_{\imath^*g}$ is extended as a bilinear form on complex forms on $\p M$, and $d\sigma_{\p M}=\imath^*(i_\nu d\Vol_g)$ is the volume form on $\p M$ induced by $d\Vol_g$.

\subsection{The normal and parallel parts of differential forms}
The outward unit normal $\nu$ to $\p M$ can be extended to a vector field near $\p M$ by parallel transport along normal geodesics (initiating from $\p M$ in the direction of $-\nu$), and then to a vector field on $M$ via a cutoff function. For $w\in C^\infty\Omega^m(M)$, we introduce
$$
\eta_\perp=\nu\wedge i_\nu \eta,\quad \eta_\parallel=\eta-\eta_\perp.
$$
Using \eqref{eqn::contraction on wedge}, one can see that $i_\nu \eta_\perp=i_\nu \eta$, so $i_\nu \eta_\parallel=0$. Since $\mathbf{t}(\nu)=0$, we also have $\mathbf{t}(\eta_\perp)=0$, so $\mathbf{t}(\eta)=\mathbf{t}(\eta_\parallel)$. It is clear that $\nu\wedge \eta_\perp=0$.

\subsection{Integration by parts}
Let us first prove the following simple result which will be used in formulating integration by parts formula in appropriate way.

\begin{Lemma}\label{lemma::expression for boundary term of Stokes' theorem}
If $\eta\in C^\infty\Omega^m(M)$ and $\zeta\in C^\infty\Omega^{m+1}(M)$, then for an open subset $\Gamma\subset \p M$ the following holds
$$
(\mathbf{t}(\eta)|\mathbf{t}(i_\nu  \zeta))_{L^2\Omega^{m}(\Gamma)}=\int_{\Gamma}\mathbf{t}(\eta\wedge*\overline\zeta).
$$
\end{Lemma}
\begin{proof}
First, we show that $\<\eta,i_\nu \zeta\>_g\,d\sigma_{\p M}=\mathbf{t}(\eta\wedge*\overline\zeta)$. Since $\<\nu\wedge\eta,\zeta\>_g=\<\eta,i_\nu \zeta\>_g$, we have
$$
\<\eta,i_\nu \zeta\>_g\,d\sigma_{\p M}=\<\nu\wedge\eta,\zeta\>_g\,d\sigma_{\p M}=\<\nu\wedge\eta,\zeta\>_g\mathbf{t}(i_\nu d\Vol_g)=\mathbf{t}(i_\nu((\nu\wedge\eta)\wedge*\overline\zeta)).
$$
Using \eqref{eqn::contraction on wedge} and $\mathbf{t}(\nu)=0$, this gives
$$
\<\eta,i_\nu \zeta\>_g\,d\sigma_{\p M}=\mathbf{t}(\eta\wedge*\overline\zeta)-\mathbf{t}(\nu)\wedge\mathbf{t}(i_\nu(\eta\wedge*\overline\zeta))=\mathbf{t}(\eta\wedge*\overline\zeta).
$$
Next, we show that $\<\eta,i_\nu \zeta\>_g=\<\mathbf{t}(\eta),\mathbf{t}(i_\nu \zeta)\>_{\imath^*g}$ on $\p M$. Indeed, observe that $(i_\nu \zeta)_\perp=0$. Therefore, $i_\nu \zeta=(i_\nu \zeta)_\parallel$ and hence on $\p M$ we get
$$
\<\eta,i_\nu \zeta\>_g=\<\eta,(i_\nu \zeta)_\parallel\>_g=\<\eta_\parallel,(i_\nu \zeta)_\parallel\>_g=\<\mathbf{t}(\eta),\mathbf{t}(i_\nu \zeta)\>_g=\<\mathbf{t}(\eta),\mathbf{t}(i_\nu \zeta)\>_{\imath^*g}.
$$
Collecting all these, we get $\<\mathbf{t}(\eta),\mathbf{t}(i_\nu \zeta)\>_{\imath^*g}\,d\sigma_{\p M}=\mathbf{t}(\eta\wedge*\overline\zeta)$. Finally, integrating over $\Gamma\subset\p M$ we get the result.
\end{proof}

For $\eta\in C^\infty\Omega^{m}(M)$ and $\zeta\in C^\infty\Omega^{m+1}(M)$, using Stokes' theorem, Lemma~\ref{lemma::expression for boundary term of Stokes' theorem} (with $\Gamma=\p M$) and \eqref{eqn::delta in terms of * and d}, we have the following integration by parts formula for $d$ and $\delta$
\begin{equation}\label{eqn::integration by parts for d and delta}
(\mathbf{t}(\eta)|\mathbf{t}(i_\nu \zeta))_{L^2\Omega^{m}(\p M)}=(d\eta|\zeta)_{L^2\Omega^{m+1}(M)}-(\eta|\delta\zeta)_{L^2\Omega^{m}(M)}.
\end{equation}

\subsection{Extensions of trace operators}\label{section::extensions of t and t i_nu}
The tangential trace operator $\mathbf{t}$ has an extension to a bounded operator from $H^1\Omega^{m}(M)$ to $H^{1/2}\Omega^{m}(\p M)$. Moreover, for every $f\in H^{1/2}\Omega^{m}(\p M)$, there is $u\in H^1\Omega^{m}(M)$ such that $\mathbf{t}(u)=f$ and
$$
\|u\|_{H^1\Omega^{m}(M)}\le C \|f\|_{H^{1/2}\Omega^{m}(\p M)};
$$
see \cite[Theorem~1.3.7]{Schwarz1995} and comments.\smallskip

Next, the operator $\mathbf{t}(i_\nu\,\cdot\,)$ is bounded from $H^1\Omega^{m}(M)$ to $H^{1/2}\Omega^{m-1}(\p M)$. Moreover, for every $h\in H^{1/2}\Omega^{m-1}(\p M)$, there is $\zeta\in H^1\Omega^{m}(M)$ such that $\mathbf{t}(i_\nu\zeta)=h$ and
$$
\|\zeta\|_{H^1\Omega^{m}(M)}\le C \|h\|_{H^{1/2}\Omega^{m-1}(\p M)}.
$$
In fact, we can take $\zeta=\nu\wedge w$, where $w\in H^1\Omega^{m-1}(M)$ such that $\mathbf{t}(w)=h$ and $\|w\|_{H^1\Omega^{m-1}(M)}\le C\|h\|_{H^{1/2}\Omega^{m-1}(\p M)}$.\smallskip

Finally, if $f\in H^{1/2}\Omega^m(\p M)$ and $h\in H^{1/2}\Omega^{m-1}(\p M)$, there is $\xi\in H^1\Omega^m(M)$ such that $\mathbf t(\xi)=f$, $\mathbf t(i_\nu \xi)=h$ and 
$$
\|\xi\|_{H^1\Omega^m(M)}\le C \|f\|_{H^{1/2}\Omega^{m}(\p M)}+C \|h\|_{H^{1/2}\Omega^{m-1}(\p M)}.
$$
This time, we can take $\xi=u_\parallel+\zeta_\perp$, where $u\in H^1\Omega^{m}(M)$ such that $\mathbf{t}(u)=f$ and $\|u\|_{H^1\Omega^{m}(M)}\le C \|f\|_{H^{1/2}\Omega^{m}(\p M)}$ and $\zeta\in H^1\Omega^{m}(M)$ such that $\mathbf{t}(i_\nu\zeta)=h$ and $\|\zeta\|_{H^1\Omega^{m}(M)}\le C \|h\|_{H^{1/2}\Omega^{m-1}(\p M)}$.




\section{Properties of $H_d\Omega^m(M)$ and $H_\delta\Omega^m(M)$ spaces}\label{section::on H_d and H_delta spaces}
Let $(M,g)$ be a compact oriented $n$-dimensional Riemannian manifold with smooth boundary. In this paper we work with the Hilbert spaces $H_d\Omega^{m}(M)$ and $H_\delta\Omega^{m}(M)$ which are the largest domains of $d$ and $\delta$, respectively, acting on $m$-forms:
\begin{align*}
H_d\Omega^{m}(M)&:=\{w\in L^2\Omega^{m}(M):dw\in L^2\Omega^{m+1}(M)\},\\
H_\delta\Omega^{m}(M)&:=\{u\in L^2\Omega^{m}(M):\delta u\in L^2\Omega^{m-1}( M)\}
\end{align*}
endowed with the inner products
\begin{align*}
(w_1|w_2)_{H_{d}\Omega^m(M)}&:=(w_1|w_2)_{L^2\Omega^{m}(M)}+(dw_1|dw_2)_{L^2\Omega^{m+1}(M)},\\
(u_1|u_2)_{H_{\delta}\Omega^m(M)}&:=(u_1|u_2)_{L^2\Omega^{m}(M)}+(\delta u_1|\delta u_2)_{L^2\Omega^{m-1}(M)}
\end{align*}
and the corresponding norms
$$
\|w\|_{H_{d}\Omega^{m}(M)}^2:=(w|w)_{H_{d}\Omega^m(M)},\quad \|u\|_{H_{\delta}\Omega^{m}(M)}^2:=(u|u)_{H_{\delta}\Omega^m(M)}.
$$
In the present section we prove some important properties of these spaces.

\subsection{Trace operators} 
In this subsection we show that there are bounded extensions $\mathbf{t}:H_d\Omega^{m}(M)\to H^{-1/2}\Omega^{m}(\p M)$ and $\mathbf{t}(i_\nu\,\cdot\,):H_\delta\Omega^{m+1}(M)\to H^{-1/2}\Omega^{m}(\p M)$.\smallskip

Let $(\cdot|\cdot)_{\p M}$ be the distributional duality on $\p M$ naturally extending $(\cdot|\cdot)_{L^2\Omega^m(\p M)}$

\begin{Proposition}\label{prop::tangential trace operator is extended to H_d}
{\rm (a)} The operator $\mathbf{t}:H^1\Omega^m(M)\to H^{1/2}\Omega^m(\p M)$ has its extension to a bounded operator $\mathbf{t}:H_d\Omega^m(M)\to H^{-1/2}\Omega^m(\p M)$ and the following integration by parts formula holds
$$
(\mathbf{t}(\eta)|\mathbf{t}(i_\nu \zeta))_{\p M}=(d\eta|\zeta)_{L^2\Omega^{m+1}(M)}-(\eta|\delta\zeta)_{L^2\Omega^{m}(M)}
$$
for all $\eta\in H_d\Omega^m(M)$ and $\zeta\in H^1\Omega^{m+1}(M)$\smallskip

{\rm (b)} The operator $\mathbf{t}(i_\nu\,\cdot\,):H^1\Omega^{m+1}(M)\to H^{1/2}\Omega^{m}(\p M)$ has its extension to a bounded operator $\mathbf{t}(i_\nu\,\cdot\,):H_\delta\Omega^{m+1}(M)\to H^{-1/2}\Omega^{m}(\p M)$ and the following integration by parts formula holds
$$
(\mathbf{t}(i_\nu \zeta)|\mathbf{t}(\eta))_{\p M}=(\zeta|d\eta)_{L^2\Omega^{m+1}(M)}-(\delta \zeta|\eta)_{L^2\Omega^{m}(M)}
$$
for all $\zeta\in H_\delta\Omega^{m+1}(M)$ and $\eta\in H^1\Omega^m(M)$.
\end{Proposition}

Now we introduce the following space on the boundary $\p M$
\begin{align*}
TH_d\Omega^{m}(\p M)&:=\{\mathbf{t}(w):w\in H_d\Omega^{m}(M)\},\\
TH_{\delta}\Omega^{m}(\p M)&:=\{\mathbf{t}(i_\nu u):u\in H_{\delta}\Omega^{m}(M)\}
\end{align*}
endowed with the norms
\begin{align*}
\|f\|_{TH_d\Omega^{m}(\p M)}&:=\inf\{\|w\|_{H_d\Omega^{m}(M)}:\mathbf{t}(w)=f,\,w\in H_d\Omega^{m}(M)\},\\
\|h\|_{TH_{\delta}\Omega^{m}(\p M)}&:=\inf\{\|u\|_{H_\delta\Omega^{m}(M)}:\mathbf{t}(u)=h,\,u\in H_{\delta}\Omega^{m}(M)\}.
\end{align*}
Then Proposition~\ref{prop::tangential trace operator is extended to H_d} implies that the operators $\mathbf{t}:H_d\Omega^{m}(M)\to TH_d\Omega^{m}(\p M)$ and $\mathbf{t}:H_{\delta}\Omega^{m}(M)\to TH_{\delta}\Omega^{m}(\p M)$ are bounded under these topologies.

\begin{proof}
Let us first prove part (a). Let $w\in C^\infty\Omega^{m}(M)$ and $f\in H^{1/2}\Omega^{m}(\p M)$. Then using integration parts formula~\eqref{eqn::integration by parts for d and delta}, we have
\begin{align*}
(\mathbf{t}(w)|f)_{L^2\Omega^{m}(\p M)}&=(\mathbf{t}(w)|\mathbf{t}(i_\nu \zeta))_{L^2\Omega^{m}(\p M)}\\
&=(dw|\zeta)_{L^2\Omega^{m+1}(M)}-(w|\delta\zeta)_{L^2\Omega^{m}(M)},
\end{align*}
where $\zeta\in H^1\Omega^{m+1}(M)$ such that $\mathbf{t}(i_\nu\zeta)=f$ and $\|\zeta\|_{H^1\Omega^{m+1}(M)}\le C \|f\|_{H^{1/2}\Omega^{m}(\p M)}$. Then
$$
|(\mathbf{t}(w)|f)_{L^2\Omega^{m}(\p M)}|\le C\|w\|_{H_d\Omega^m(M)}\|\zeta\|_{H^1\Omega^{m+1}(M)}\le C\|w\|_{H_d\Omega^m(M)}\|f\|_{H^{1/2}\Omega^{m}(\p M)}.
$$
Therefore, $\mathbf{t}$ can be extended to a bounded operator $H_d\Omega^m(M)\to H^{-1/2}\Omega^m(\p M)$. In fact, if $\eta\in H_d\Omega^m(M)$, then we define $\mathbf{t}(\eta)$ as
$$
(\mathbf{t}(\eta)|\mathbf{t}(i_\nu \zeta))_{\p M}=(d\eta|\zeta)_{L^2\Omega^{m+1}(M)}-(\eta|\delta\zeta)_{L^2\Omega^{m}(M)},
$$
where $\zeta\in H^1\Omega^{m+1}(M)$.\smallskip

Now we prove part (b). Let $w\in C^\infty\Omega^{m+1}(M)$ and $f\in H^{1/2}\Omega^{m}(\p M)$. Then using integration parts formula~\eqref{eqn::integration by parts for d and delta}, we have
\begin{align*}
(\mathbf{t}(i_\nu w)|f)_{L^2\Omega^{m}(\p M)}&=(\mathbf{t}(i_\nu w)|\mathbf{t}(u))_{L^2\Omega^{m}(\p M)}\\
&=(w|du)_{L^2\Omega^{m+1}(M)}-(\delta w|u)_{L^2\Omega^{m}(M)},
\end{align*}
where $u\in H^1\Omega^{m}(M)$ such that $\mathbf{t}(u)=f$ and $\|u\|_{H^1\Omega^{m}(M)}\le C \|f\|_{H^{1/2}\Omega^{m}(\p M)}$. Therefore, we can estimate
$$
|(\mathbf{t}(i_\nu w)|f)_{L^2\Omega^{m}(\p M)}|\le C\|w\|_{H_\delta \Omega^{m+1}(M)}\|\zeta\|_{H^1\Omega^{m}(M)}\le C\|w\|_{H_\delta\Omega^{m+1}(M)}\|f\|_{H^{1/2}\Omega^{m}(\p M)}.
$$
Thus, $\mathbf{t}(i_\nu\,\cdot\,)$ can be extended to a bounded operator $H_\delta\Omega^{m+1}(M)\to H^{-1/2}\Omega^m(\p M)$. In fact, if $\zeta\in H_\delta\Omega^{m+1}(M)$ we define $\mathbf{t}(i_\nu \zeta)$ as
$$
(\mathbf{t}(i_\nu \zeta)|\mathbf{t}(\eta))_{\p M}=(\zeta|d\eta)_{L^2\Omega^{m+1}(M)}-(\delta \zeta|\eta)_{L^2\Omega^{m}(M)},
$$
where $\eta\in H^1\Omega^m(M)$.
\end{proof}

\subsection{Embedding property}

We will also need the following embedding result.
\begin{Proposition}\label{prop::bounded embedding of H_{d,delta} with regular boundary value into H^1}
If $u\in H_{d}\Omega^m(M)\cap H_\delta\Omega^m(M)$ with $\mathbf{t}(u)\in H^{1/2}\Omega^m(\p M)$%
, then $u\in H^1\Omega^m(M)$ and
$$\|u\|_{H^1\Omega^m(M)}\le C\big(\|u\|_{H_d\Omega^m(M)}+\|\delta u\|_{L^2\Omega^{m-1}(M)}+\|{\mathbf{t}}(u)\|_{H^{1/2}\Omega^m(\p M)}\big)\\
$$
for some constant $C>0$ independent of $u$.
\end{Proposition}

In Euclidean setting, this was proven in the case $m=1$ by Costabel \cite{Costabel1990}; see also \cite{KirschHettlich2015,Monk2003}. Here we give a new proof, which can be carried out over manifolds and for arbitrary $m$. Our proof is based on the following result from \cite{Schwarz1995}. We write
$$
\mathcal H^{m}_D(M):=\{u\in H^1\Omega^m(M):du=0,\quad \delta u=0,\quad \mathbf t(u)=0\}.
$$

\begin{Lemma}\label{lemma::from Schwarz thm3.2.5}
Let $k\ge 0$ be an integer. Given $w\in H^k\Omega^{m+1}(M)$, $v\in H^k\Omega^{m-1}(M)$ and $h\in H^{k+1}\Omega^m(M)$, there is a unique $\psi\in H^{k+1}\Omega^m(M)$, up to a form in $\mathcal H_D^m(M)$, that solves
$$
d\psi=w,\quad \delta \psi=v,\quad \mathbf{t}(\psi)=\mathbf{t}(h)
$$
if and only if
$$
dw=0,\quad \mathbf t(w)=\mathbf t(dh),\quad \delta v=0
$$
and
$$
(w|\chi)_{L^2\Omega^{m+1}(M)}=(\mathbf t(h)|\mathbf t(i_\nu \chi))_{L^2\Omega^m(\p M)},\quad (v|\lambda)_{L^2\Omega^{m-1}(M)}=0
$$
for all $\chi\in\mathcal H^{m+1}_D(M)$, $\lambda\in\mathcal H^{m-1}_D(M)$. Moreover, $\psi$ satisfies the estimate
\begin{align*}
\|\psi\|_{H^{k+1}\Omega^m(M)}\le &C\big(\|w\|_{H^k\Omega^{m+1}(M)}+\|v\|_{H^k\Omega^{m-1}(M)}\big)\\
&+C\big(\|\mathbf t(h)\|_{H^{k+1/2}\Omega^m(\p M)}+\|\mathbf t(*h)\|_{H^{k+1/2}\Omega^{n-m}(\p M)}\big).
\end{align*}
\end{Lemma}
\begin{proof}
Follows from \cite[Theorem~3.2.5]{Schwarz1995}.
\end{proof}

\begin{proof}[Proof of Proposition~\ref{prop::bounded embedding of H_{d,delta} with regular boundary value into H^1}]
For a given $u\in H_{d}\Omega^m(M)\cap H_\delta\Omega^m(M)$, write $w=du\in L^2\Omega^{m+1}(M)$ and $v=\delta u\in L^2\Omega^{m-1}(M)$. Since $\mathbf{t}(u)\in H^{1/2}\Omega^m(\p M)$, by discussion in Section~\ref{section::extensions of t and t i_nu} there is $h\in H^1\Omega^m(M)$ such that $\mathbf{t}(h)=\mathbf{t}(u)$, $\mathbf t(i_\nu h)=0$ and
\begin{equation}\label{eqn::estimate for the trace of h in the proof of embedding result}
\|h\|_{H^1\Omega^m(M)}\le C\|\mathbf t(h)\|_{H^{1/2}\Omega^m(\p M)}=C\|\mathbf t(u)\|_{H^{1/2}\Omega^m(\p M)}.
\end{equation}
We wish to use Lemma~\ref{lemma::from Schwarz thm3.2.5}, and hence we need to show that $w$, $v$ and $h$ satisfy the hypothesis of Lemma~\ref{lemma::from Schwarz thm3.2.5}. Obviously, we have $dw=0$ and $\delta v=0$. Integrating by parts and using that $\mathbf t(u)=\mathbf t(h)$, we can show that for all $\chi\in\mathcal H^{m+1}_D(M)$
$$
(w|\chi)_{L^2\Omega^{m+1}(M)}=(du|\chi)_{L^2\Omega^{m+1}(M)}=(\mathbf{t}(h)|\mathbf{t}(i_\nu\chi))_{L^2\Omega^m(\p M)}.
$$
Similary for all $\lambda\in\mathcal H^{m-1}_D(M)$, using the integration by parts formula in part (b) of Proposition~\ref{prop::tangential trace operator is extended to H_d}, we can show that
$$
(v|\lambda)_{L^2\Omega^{m-1}(M)}=(\delta u|\lambda)_{L^2\Omega^{m-1}(M)}=-(\mathbf t(i_\nu u)|\mathbf t(\lambda))_{\p M}=0.
$$
Next, we show that $\mathbf t(w)=\mathbf t(dh)$. For arbitrary $\varphi\in H^{1/2}\Omega^{m+1}(\p M)$, as discussed in Section~\ref{section::extensions of t and t i_nu}, there is $\zeta\in H^1\Omega^{m+2}(M)$ such that $\mathbf t(i_\nu \zeta)=\varphi$. Then, using integration by parts formulas in Proposition~\ref{prop::tangential trace operator is extended to H_d}, we get
$$
(\mathbf t(w)|\varphi)_{\p M}=(\mathbf t(du)|\mathbf t(i_\nu \zeta))_{\p M}=-(du|\delta\zeta)_{L^2\Omega^{m+1}(M)}=-(\mathbf t(u)|\mathbf t(i_\nu \delta \zeta))_{\p M}.
$$
Since $\mathbf t(u)=\mathbf t(h)$, using integration by parts formulas in Proposition~\ref{prop::tangential trace operator is extended to H_d}, gives
$$
(\mathbf t(w)|\varphi)_{\p M}=-(\mathbf t(h)|\mathbf t(i_\nu \delta \zeta))_{\p M}=-(dh|\delta\zeta)_{L^2\Omega^{m+1}(M)}=(\mathbf t(dh)|\varphi)_{\p M},
$$
which implies $\mathbf t(w)=\mathbf t(dh)$. Therefore, applying Lemma~\ref{lemma::from Schwarz thm3.2.5} we find $\psi\in H^1\Omega^m(M)$ such that $d\psi=w$, $\delta\psi=v$ and ${\mathbf{t}}(\psi)={\mathbf{t}}(h)={\mathbf{t}}(u)$ and satisfying
\begin{align*}
\|\psi\|_{H^{1}\Omega^m(M)}\le &C\big(\|w\|_{L^2\Omega^{m+1}(M)}+\|v\|_{L^2\Omega^{m-1}(M)}\big)\\
&+C\big(\|\mathbf t(u)\|_{H^{1/2}\Omega^m(\p M)}+\|\mathbf t(*h)\|_{H^{1/2}\Omega^{n-m}(\p M)}\big).
\end{align*}
Using boundedness of $\mathbf t:H^1\Omega^{n-m}(M)\to H^{1/2}\Omega^{n-m}(\p M)$ and \eqref{eqn::estimate for the trace of h in the proof of embedding result},
$$
\|\mathbf t(*h)\|_{H^{1/2}\Omega^{n-m}(\p M)}\le C\|*h\|_{H^1\Omega^{n-m}(M)}\le C\|h\|_{H^1\Omega^{m}(M)}\le C\|\mathbf t(u)\|_{H^{1/2}\Omega^{m}(\p M)}.
$$
Therefore, $\psi$ satisfies the estimate
$$
\|\psi\|_{H^{1}\Omega^m(M)}\le C\big(\|w\|_{L^2\Omega^{m+1}(M)}+\|v\|_{L^2\Omega^{m-1}(M)}+\|\mathbf t(u)\|_{H^{1/2}\Omega^m(\p M)}\big).
$$
Write $\rho=u-\psi$, then $d\rho=0$ and $\delta \rho=0$. Therefore, $\rho$ solves $-\Delta \rho=0$ with $\mathbf t(\rho)=0$, $\mathbf t(\delta \rho)=0$. By \cite[Theorem~2.2.4]{Schwarz1995}, it follows that $\rho=0$. This clearly implies the result.
\end{proof}

\subsection{Density properties}
In this subsection we prove the following two results regarding the density of $C^\infty\Omega^m(M)$ in both $H_d\Omega^{m}(M)$ and $H_\delta\Omega^{m}(M)$.
\begin{Proposition}\label{prop::density of smooth forms in H_delta}
The space $C^\infty\Omega^m(M)$ is dense in $H_\delta\Omega^{m}(M)$.
\end{Proposition}

\begin{proof}
The statement is equivalent to showing that if $u\in H_\delta\Omega^m(M)$ is orthogonal to $C^\infty\Omega^m(M)$ in $H_\delta \Omega^m(M)$-inner product, then $u=0$. Suppose that
\begin{equation}\label{eqn::u is perp to smooth forms in H_delta}
(u|\phi)_{H_\delta \Omega^m(M)}=(u|\phi)_{L^2\Omega^m(M)}+(\delta u|\delta\phi)_{L^2\Omega^{m-1}(M)}=0,\quad \phi\in C^\infty\Omega^m(M).
\end{equation}
Let $\widetilde M$ be a compact manifold with smooth boundary such that $M\subset\subset\widetilde M^{\rm int}$ and let by $g$ on $\widetilde M$ we denote a smooth extension of $g$ from $M$ to $\widetilde M$. Let $\widetilde u$ and $\widetilde{\delta u}$ denote the extensions of $u$ and $\delta u$ to $\widetilde M$ by zero. It is clear that $\widetilde u\in L^2\Omega^m(\widetilde M)$ and $\widetilde{\delta u}\in L^2\Omega^{m-1}(\widetilde M)$. By \eqref{eqn::u is perp to smooth forms in H_delta}, $\widetilde u$ and $\widetilde{\delta u}$ satisfy
$$
(\widetilde u|\phi)_{L^2\Omega^m(\widetilde M)}+(\widetilde{\delta u}|\delta\phi)_{L^2\Omega^{m-1}(\widetilde M)}=0,\quad \phi\in C^\infty_0\Omega^m(\widetilde M^{\rm int}).
$$
This in particular implies that $\widetilde u=-d\widetilde{\delta u}$. Since $\widetilde u\in L^2\Omega^m(\widetilde M)$, we have $\widetilde{\delta u}\in H_{d,0}\Omega^{m-1}(\widetilde M)$. Therefore, $\delta u=\widetilde{\delta u}|_{M}\in H_{d}\Omega^{m-1}(M)\cap H_\delta \Omega^{m-1}(M)$. Since $\widetilde{\delta u}=0$ in $\widetilde M\setminus M$, we have $\mathbf{t}(\delta u)=\mathbf{t}(\widetilde{\delta u})=0$ on $\p M$. Then by Proposition~\ref{prop::bounded embedding of H_{d,delta} with regular boundary value into H^1}, $\delta u\in H^1\Omega^{m-1}_D(M)$. There is a sequence $\{\phi_k\}_{k=1}^\infty\subset C^\infty_0\Omega^{m-1}(M^{\rm int})$ such that $\|\delta u-\phi_k\|_{H^1\Omega^{m-1}(M)}\to 0$ as $k\to\infty$. Note also that, in particular, \eqref{eqn::u is perp to smooth forms in H_delta} gives $u=d\delta u$. Using all these facts, we can show that
\begin{align*}
(u|u)_{L^2\Omega^m(M)}+(\delta u|\delta u)_{L^2\Omega^{m-1}(M)}&=(u|d\delta u)_{L^2\Omega^m(M)}+(\delta u|\delta u)_{L^2\Omega^{m-1}(M)}\\
&=\lim_{k\to\infty}\big[(u|d\phi_k)_{L^2\Omega^m(M)}+(\delta u|\phi_k)_{L^2\Omega^{m-1}(M)}\big]\\
&=\lim_{k\to\infty}\big[(d\delta u|d\phi_k)_{L^2\Omega^m(M)}+(\delta u|\phi_k)_{L^2\Omega^{m-1}(M)}\big].
\end{align*}
Integrating by parts and using \eqref{eqn::u is perp to smooth forms in H_delta}, we get
\begin{align*}
(u|u)_{L^2\Omega^m(M)}&+(\delta u|\delta u)_{L^2\Omega^{m-1}(M)}\\
&=\lim_{k\to\infty}\big[(\delta u|\delta d\phi_k)_{L^2\Omega^m(M)}+(u|d\phi_k)_{L^2\Omega^{m-1}(M)}\big]=0.
\end{align*}
This implies $u=0$ as desired.
\end{proof}

\begin{Proposition}\label{prop::density of smooth forms in H_d}
The space $C^\infty\Omega^m(M)$ is dense in $H_d\Omega^{m}(M)$.
\end{Proposition}

\begin{proof}
This follows from Proposition~\ref{prop::density of smooth forms in H_delta} using the fact that the Hodge star operator $*$ is an isometry between $H_d\Omega^{m}(M)$ and $H_\delta\Omega^{n-m}(M)$.
\end{proof}


\section{Helmholtz decompositions and compact embedding results}\label{section::well-posedness}
Let $(M,g)$ be a compact oriented $n$-dimensional Riemannian manifold with smooth boundary. Throughout  the section we assume that $\alpha\in L^\infty(M)$ such that $\Re\alpha\ge c$ for some constant $c>0$. 

\subsection{Helmholtz decompositions of $H_{d}\Omega^1(M)$, $H_{d,0}\Omega^1(M)$ and $L^2\Omega^1(M)$} For the proof of Theorem~\ref{thm::well posedness new version}, we will use Helmholtz type decomposition of $H_{d,0}\Omega^1(M)$ and $L^2\Omega^1(M)$ suitable for Maxwell's equations. For the proofs we closely follow \cite{Monk2003}, see also \cite{KirschHettlich2015}.\smallskip

Define the spaces
\begin{align*}
L^2\Omega^1(M)_{0,\alpha}:&=\{w\in L^2\Omega^1(M): (\alpha w|dh)_{L^2\Omega^1(M)}=0,\,\, h\in H^1_0(M)\},\\
H_{d}\Omega^1(M)_\alpha:&=\{w\in H_{d}\Omega^1(M): (\alpha w|\varphi)_{L^2\Omega^1(M)}=0,\,\, \varphi\in H_{d}(0,\Omega^1(M))\},\\
H_{d,0}\Omega^1(M)_{0,\alpha}:&=\{w\in H_{d,0}\Omega^1(M): (\alpha w|dh)_{L^2\Omega^1(M)}=0,\,\, h\in H^1_0(M)\}.
\end{align*}
\begin{Proposition}\label{prop::Helmholtz decomposition of H_{d,0}}
The space $dH_0^1(M)=\{dh\in L^2\Omega^1(M):h\in H^1_0(M)\}$ is closed in $L^2\Omega^1(M)$ and in $H_{d,0}\Omega^1(M)$, and the following orthogonal decompositions hold
\begin{align}
L^2\Omega^1(M)&=L^2\Omega^1(M)_{0,\alpha}\oplus dH_0^1(M),\label{eqn::decomposition of L^2}\\
H_{d,0}\Omega^1(M)&=H_{d,0}\Omega^1(M)_{0,\alpha}\oplus dH_0^1(M),\label{eqn::decomposition of H_d0}
\end{align}
where all of the projection operators are bounded. Moreover, the projection of $H_{d,0}\Omega^1(M)$ onto $H_{d,0}\Omega^1(M)_{0,\alpha}$ is the restriction of the projection of $L^2\Omega^1(M)$ onto $L^2\Omega^1(M)_{0,\alpha}$.
\end{Proposition}

\begin{proof}
To prove closedness of $dH_0^1(M)$ in $L^2\Omega^1(M)$, consider a sequence $\{h_k\}_{k=1}^\infty\subset H^1_0(M)$ such that $\|dh_k-u\|_{L^2\Omega^1(M)}\to 0$ as $k\to\infty$ for some $u\in L^2\Omega^1(M)$. In particular, $\{dh_k\}_{k=1}^\infty$ is a Cauchy sequence in $L^2\Omega^1(M)$. Then by Poincar\'e inequality, $\{h_k\}_{k=1}^\infty$ is a Cauchy sequence in $L^2(M)$. Hence, $\{h_k\}_{k=1}^\infty$ is a Cauchy sequence in $H^1(M)$. Therefore, $u=dh$ for some $h\in H^1(M)$. Finally, by closedness of $H_0^1(M)$ in $H^1(M)$, we have $h\in H^1_0(M)$.\smallskip

Next, to prove closedness of $dH_0^1(M)$ in $H_{d,0}\Omega^1(M)$, consider a sequence $\{h_k\}_{k=1}^\infty\subset H^1_0(M)$ such that $\|dh_k-u\|_{L^2\Omega^1(M)}\to 0$ as $k\to\infty$ for some $u\in H_{d,0}\Omega^1(M)$. In particular, $\{dh_k\}_{k=1}^\infty$ is a Cauchy sequence in $L^2\Omega^1(M)$. Since $dh_k=0$ for all $k\ge 1$, $\{dh_k\}_{k=1}^\infty$ is a Cauchy sequence in $L^2(M)$. 
Therefore, $u=dh$ for some $h\in H^1(M)$. Finally, by closedness of $dH_0^1(M)$ in $L^2\Omega^1(M)$, we have $h\in H^1_0(M)$.\smallskip

To prove (\ref{eqn::decomposition of L^2} -- \ref{eqn::decomposition of H_d0}), consider the sesquilinear form $A$ on $dH_{0}^1(M)$ defined as
$$
A(dh,dh')=(\alpha dh|dh')_{L^2\Omega^1(M)},\quad h,h'\in H_0^1(M).
$$
It is clear that
$$
|A(dh,dh')|\le C\|dh\|_{L^2\Omega^1(M)}\|dh'\|_{L^2\Omega^1(M)}
$$
and that 
$$
\Re A(dh,dh)=(\Re(\alpha) dh|dh)_{L^2\Omega^1(M)} \ge c\|dh\|_{L^2\Omega^1(M)}^2.
$$
Thus, the form $A$ is strictly coercive on $dH_{0}^1(M)$. For a given $e\in L^2\Omega^1(M)$, consider the bounded linear functional $\ell_e:dH^1_0(M)\to \C$ defined as
$$
\ell_e(dh')=(\alpha e|dh')_{L^2\Omega^1(M)}.
$$
Applying the Lax-Milgram's lemma (see~e.g.~\cite[Lemma~2.21]{Monk2003}), we obtain a bounded linear operator $G:L^2\Omega^1(M)\to H_0^1(M)$ such that
$$
\ell_e(dh')=A(Ge,dh'),\quad e\in L^2\Omega^1(M),\quad h'\in H_0^1(M).
$$
This implies that
\begin{equation}\label{eqn::e-dGe is in H_{d,0,varepsilon} in a weak sense}
(\alpha(e-dGe)|dh')_{L^2\Omega^1(M)}=0,\quad h'\in H_0^1(M),
\end{equation}
and hence $e-dGe\in L^2\Omega^1(M)_{0,\alpha}$.\smallskip



Thus, we can claim that every $e\in L^2\Omega^1(M)$ can be uniquely decomposed as $e=e_0+dh$ where $e_0=(e-dGe)\in L^2\Omega^1(M)_{0,\alpha}$ and $h=Ge\in H^1_0(M)$. Hence, we have shown~\eqref{eqn::decomposition of L^2}.\smallskip

If $e\in H_{d,0}\Omega^1(M)$, then $e_0=e-dGe\in H_{d,0}\Omega^1(M)$ since
$$
\mathbf t(e_0)=\mathbf t(e)-\mathbf t(dGe)=-d_{\p M}(Ge)|_{\p M}=0.
$$
From \eqref{eqn::e-dGe is in H_{d,0,varepsilon} in a weak sense} we also can see that $e_0 \in H_{d,0}\Omega^1(M)_{0,\alpha}$. This gives the decomposition~\eqref{eqn::decomposition of H_d0}.
\end{proof}

It is easy to see that closedness of $dH_0^1(M)$ in $L^2\Omega^1(M)$ imply closedness of the former in $H_d\Omega^1(M)$. Moreover, the sesquilinear form $A$ in the proof of Proposition~\ref{prop::Helmholtz decomposition of H_{d,0}} can be defined on $H_{d}(0,\Omega^1(M))$; see the definition of the latter space below. The same is true for the linear functional $\ell_e$. Furthermore, the latter makes sense even for $e\in L^2\Omega^1(M)$. Therefore, the similar arguments, but $L^2\Omega^1(M)$ replaced by $H_{d,0}\Omega^1(M)$ and $dH_0^1(M)$ replaced by $H_{d}(0,\Omega^1(M))$, imply the following result.

\begin{Proposition}\label{prop::Helmholtz decomposition of H_{d}}
The space $dH_0^1(M)$ is closed in $H_{d}\Omega^1(M)$ and the following orthogonal decomposition holds
\begin{equation}\label{eqn::decomposition of H_d}
H_d\Omega^1(M)=H_{d}\Omega^1(M)_\alpha\oplus H_{d}(0,\Omega^1(M)),
\end{equation}
where
$$
H_{d}(0,\Omega^1(M))=\{\varphi\in H_d\Omega^1(M):d\varphi=0\}
$$
and all of the projection operators are bounded.
\end{Proposition}

\subsection{Compact embedding results}

We will also need the following results on compact embedding of $H_{d,0}\Omega^1(M)\cap H_\delta\Omega^1(M)$ and $H_{d,0}\Omega^1(M)_{0,\alpha}$ into $L^2\Omega^1(M)$.

\begin{Proposition}\label{corol::compact embedding of H_{d,delta,0} into L^2}
The inclusion $H_{d,0}\Omega^1(M)\cap H_\delta\Omega^1(M)\inclusion L^2\Omega^1(M)$ is compact
\end{Proposition}

\begin{proof}
Follows from Proposition~\ref{prop::bounded embedding of H_{d,delta} with regular boundary value into H^1} and the compactness of the embedding
$$
H^1\Omega^1(M)\inclusion L^2\Omega^1(M),
$$
see e.g. \cite[Theorem~1.3.6]{Schwarz1995}.
\end{proof}

The following compact embedding result is originally due to Weber~\cite{Weber1980} in Euclidean case.

\begin{Proposition}\label{prop::compact embedding of H_{d,0,varepsilon} into L^2}
The inclusion $H_{d,0}\Omega^1(M)_{0,\alpha}\inclusion L^2\Omega^1(M)$ is compact.
\end{Proposition}
\begin{proof}
We prove this result following \cite[Proposition~2.28]{CaorsiFernandesRaffetto2000}. Consider a bounded sequence $\{u_k\}_{k=1}^\infty\subset H_{d,0}\Omega^1(M)_{0,\alpha}$. Using the Helmholtz decomposition in \eqref{eqn::decomposition of H_d0} for $\alpha=1$, we can write each $u_k$ uniquely as $u_k=u_{0,k}^1+dh^1_k$, where $u_{0,k}^1\in H_{d,0}\Omega^1(M)_{0,1}$ and $h^1_k\in H^1_0(M)$. Since $(u_k|dh^1_k)_{L^2\Omega^1(M)}=(dh^1_k|dh^1_k)_{L^2\Omega^1(M)}$, we have $\|dh^1_k\|_{H_d\Omega^1(M)}\le \|u_k\|_{H_d\Omega^1(M)}$ and hence
$$
\|u^1_{0,k}\|_{H_d\Omega^1(M)}\le C \|u_k\|_{H_d\Omega^1(M)}.
$$
Thus, the sequence $\{u_{0,k}^1\}_{k=1}^\infty\subset H_{d,0}\Omega^1(M)_{0,1}$ is bounded. Since $H_{d,0}\Omega^1(M)_{0,1}\subset H_{d,0}\Omega^1(M)\cap H_\delta\Omega^1(M)$, Proposition~\ref{corol::compact embedding of H_{d,delta,0} into L^2} implies that there is $u\in L^2\Omega^1(M)$ and a subsequence $\{u_{0,k'}^1\}_{k'=1}^\infty$ such that
\begin{equation}\label{eqn::estimate u-u_{0,k'}^1 in L^2}
\|u-u_{0,k'}^1\|_{L^2\Omega^1(M)}\to 0\quad\text{as}\quad k'\to\infty.
\end{equation}
Now, using the Helmholtz decomposition in \eqref{eqn::decomposition of L^2}, we can write $u$ uniquely as $u=u^\alpha+dh^\alpha$, where $u^\alpha\in L^2\Omega^1(M)_{0,\alpha}$ and $h^\alpha\in H^1_0(M)$. Then
\begin{align*}
(\alpha(u^\alpha-u_{k'})|(u^\alpha-u_{k'}))_{L^2\Omega^1(M)}&=(\alpha(u^\alpha-u_{k'})|(u^\alpha+dh^\alpha-u_{k'}+dh_{k'}^1))_{L^2\Omega^1(M)}\\
&=(\alpha(u^\alpha-u_{k'})|(u-u_{0,k'}^1))_{L^2\Omega^1(M)}.
\end{align*}
Together with \eqref{eqn::estimate u-u_{0,k'}^1 in L^2} this gives that
$$
\|u^\alpha-u_{k'}\|_{L^2\Omega^1(M)}\le C\|u-u_{0,k'}^1\|_{L^2\Omega^1(M)}\to 0\quad\text{as}\quad k'\to\infty.
$$
Thus, the subsequence $\{u_{k'}\}_{k'=1}^\infty$ converges to $u^\alpha$ in $L^2\Omega^1(M)$. The proof is complete.
\end{proof}

\section{Proof of Theorem~\ref{thm::well posedness new version}}\label{section::proof of thm2}

For the proof, we follow the standard variational-methods used in \cite{Costabel1991,KirschHettlich2015,Leis1968,Monk2003}. Substituting the second equation of \eqref{eqn::Maxwell in appendix} into the first equation of \eqref{eqn::Maxwell in appendix}, we obtain the following second-order equation
\begin{equation}\label{eqn::second order equation}
\delta(\mu^{-1}dE)-\omega^2\varepsilon E=i\omega J_e+*d(\mu^{-1}J_m).
\end{equation}
If we find a unique solution $E\in H_{d,0}\Omega^1(M)$ of this equation satisfying
$$
\|E\|_{H_{d}\Omega^1(M)}\le C(\|J_e\|_{L^2\Omega^1(M)}+\|J_m\|_{L^2\Omega^1(M)}),
$$
defining $H=-i\omega^{-1}\mu^{-1}(*dE-J_m)$ we obtain a unique $(E,H)\in H_{d,0}\Omega^1(M)\times H_{d}\Omega^1(M)$ solving the Maxwell equations \eqref{eqn::Maxwell in appendix} and hence satisfying
$$
\|E\|_{H_{d}\Omega^1(M)}+\|H\|_{H_{d}\Omega^1(M)}\le C(\|J_e\|_{L^2\Omega^1(M)}+\|J_m\|_{L^2\Omega^1(M)}).
$$

Therefore, the problem is reduced to finding a unique $E\in H_{d,0}\Omega^1(M)$ such that
\begin{equation}\label{eqn::second order equation in a weak form}
\begin{aligned}
(\mu^{-1}dE|de')_{L^2\Omega^2(M)}-(&\omega^2\varepsilon E|e')_{L^2\Omega^1(M)}\\
&=(i\omega J_e|e')_{L^2\Omega^1(M)}+(\mu^{-1}*J_m|de')_{L^2\Omega^2(M)}
\end{aligned}
\end{equation}
for all $e'\in H_{d,0}\Omega^1(M)$.\smallskip

Using \eqref{eqn::decomposition of H_d0}, we can decompose $E$ uniquely as $E=E_0+dh$, where $E_0\in H_{d,0}\Omega^1(M)_{0,\varepsilon}$ and $h\in H^1_0(M)$. Since $i\omega\varepsilon^{-1}J_e\in L^2\Omega^1(M)$, this can be uniquely decomposed as $i\omega\varepsilon^{-1}J_e=J_{e,0}+dj_e$, where $J_{e,0}\in L^2\Omega^1(M)_{0,\varepsilon}$ and $j_e\in H^1_0(M)$. We note here that
\begin{equation}\label{eqn::estimate for j_e in H^1 norm}
\|j_e\|_{H^1(M)}\le C\|J_e\|_{L^2\Omega^1(M)}.
\end{equation}
Using these decompositions, \eqref{eqn::second order equation in a weak form} can be written as
\begin{equation}\label{eqn::second order equation in a weak form -- rewritten}
\begin{aligned}
(\mu^{-1}dE_0|&de')_{L^2\Omega^2(M)}-(\omega^2\varepsilon E_0|e')_{L^2\Omega^1(M)}-(\omega^2\varepsilon dh|e')_{L^2\Omega^1(M)}\\
&=(\varepsilon J_{e,0}|e')_{L^2\Omega^1(M)}+(\varepsilon dj_e|e')_{L^2\Omega^1(M)}+(\mu^{-1}*J_m|de')_{L^2\Omega^2(M)}
\end{aligned}
\end{equation}
for all $e'\in H_{d,0}\Omega^1(M)$.\smallskip

Our first step is to extract $h$ from \eqref{eqn::second order equation in a weak form -- rewritten}. For this, use $e'=dh'$ for arbitrary $h'\in H^1_0(M)$ in \eqref{eqn::second order equation in a weak form -- rewritten}. Since $E_0\in H_{d,0}\Omega^1(M)_{0,\varepsilon}$ and $J_{e,0}\in L^2\Omega^1(M)_{0,\varepsilon}$, we obtain
$$
-(\omega^2\varepsilon dh|dh')_{L^2\Omega^1(M)}=(\varepsilon dj_e|dh')_{L^2\Omega^1(M)}
$$
for all $h'\in H_0^1(M)$. We rewrite this as
$$
(\varepsilon d(\omega^2h+j_e)|dh')_{L^2\Omega^1(M)}=0
$$
and take $h'=\omega^2h+j_e$. Then we obtain $h'=0$, which implies that $h=-\omega^{-2}j_e$.\smallskip

Now, we use $h=-\omega^{-2}j_e$ in \eqref{eqn::second order equation in a weak form -- rewritten} and get
\begin{align*}
(\mu^{-1}dE_0|de')_{L^2\Omega^2(M)}-(\omega^2&\varepsilon E_0|e')_{L^2\Omega^1(M)}\\
&=(\varepsilon J_{e,0}|e')_{L^2\Omega^1(M)}+(\mu^{-1}*J_m|de')_{L^2\Omega^2(M)}
\end{align*}
for all $e'\in H_{d,0}\Omega^1(M)$. Thus, our next step is to find a unique $E_0\in H_{d,0}\Omega^1(M)_{0,\varepsilon}$ satisfying
\begin{equation}\label{eqn::second order equation -- rewritten}
\delta(\mu^{-1}dE_0)-\omega^2\varepsilon E_0=\varepsilon J_{e,0}+\delta(\mu^{-1}*J_m).
\end{equation}
To solve this equation, we need the following result on existence of a solution operator
\begin{Proposition}\label{prop::resonances}
There are a constant $\lambda>0$ and a bounded linear map $T_\lambda:(H_{d,0}\Omega^1(M))'\to H_{d,0}\Omega^1(M)$ such that
\begin{equation}\label{eqn::T_lambda is an inverse of a certain PDoperator}
\delta(\mu^{-1}dT_\lambda u)+\lambda\varepsilon T_\lambda u=u,\quad u\in (H_{d,0}\Omega^1(M))'
\end{equation}
and
$$
T_\lambda(\delta(\mu^{-1}de)+\lambda \varepsilon e)=e,\quad e\in H_{d,0}\Omega^1(M).
$$
Further, if $\<u,dh'\>_M=0$ for all $h'\in H^1_0(M)$, then $T_\lambda u\in H_{d,0}\Omega^1(M)_{0,\varepsilon}$. Moreover, if $\varepsilon$ and $\mu$ are positive, then $T_\lambda|_{L^2\Omega^1(M)}$ is self-adjoint with respect to the $L^2\Omega^1(M)$-inner product.
\end{Proposition}

Here and in what follows, $\<\cdot,\cdot\>_{M}$ is the duality between $(H_{d,0}\Omega^1(M))'$ and $H_{d,0}\Omega^1(M)$ naturally extending the $L^2\Omega^1(M)$-inner product.

\begin{proof}
Consider the bilinear form on $H_{d,0}\Omega^1(M)$
$$
B(e,e'):=(\mu^{-1}de|de')_{L^2\Omega^2(M)},\quad e,e'\in H_{d,0}\Omega^1(M).
$$
Then
$$
|B(e,e')|\le C\|e\|_{H_{d,0}\Omega^1(M)}\|e'\|_{H_{d,0}\Omega^1(M)}.
$$
It is also easy to see that
$$
\Re B(e,e)\ge C_0\|de\|_{L^2\Omega^2(M)}^2\ge c_0\|e\|_{H_d\Omega^1(M)}^2-C_0\|e\|_{L^2\Omega^1(M)}^2
$$
for some constants $c_0,C_0>0$ independent of $e$. Thus, there is constant $\lambda>0$ such that the form $B(e,e')+(\lambda \varepsilon e|e')_{L^2\Omega^1(M)}$ is strictly coercive on $H_{d,0}\Omega^1(M)$. In fact, we can take $\lambda>0$ satisfying $\lambda\ge C_0/\min_M\Re(\varepsilon)$. Applying the Lax-Milgram's lemma, we obtain a bounded linear operator $T_\lambda:(H_{d,0}\Omega^1(M))'\to H_{d,0}\Omega^1(M)$ such that
\begin{equation}\label{eqn::T_lambda in a weak sense}
(\mu^{-1}dT_\lambda u|de')_{L^2\Omega^2(M)}+(\lambda \varepsilon T_\lambda u|e')_{L^2\Omega^1(M)}=\<u,e'\>_{M}
\end{equation}
for all $u\in (H_{d,0}\Omega^1(M))'$ and $e'\in H_{d,0}\Omega^1(M)$, where $\<\cdot,\cdot\>_{M}$ is the duality between $(H_{d,0}\Omega^1(M))'$ and $H_{d,0}\Omega^1(M)$. Thus, $T_\lambda$ is the operator which maps $u\in (H_{d,0}\Omega^1(M))'$ to the unique solution $e\in H_{d,0}\Omega^1(M)$ of $\delta(\mu^{-1}de)+\lambda\varepsilon e=u$.\smallskip

In particular, if $\<u,dh'\>_{M}=0$ for all $h'\in H^1_0(M)$, setting $e'=dh'$ in \eqref{eqn::T_lambda in a weak sense} we get $(\varepsilon T_\lambda u|dh')_{L^2\Omega^1(M)}=0$ and hence $T_\lambda u\in H_{d,0}\Omega^1(M)_{0,\varepsilon}$.\smallskip

To prove that $T_\lambda$ is self-adjoint, suppose $\varphi,\varphi'\in L^2\Omega^1(M)$. Then
\begin{align*}
(T_\lambda \varphi|\varphi')_{L^2\Omega^1(M)}&=(T_\lambda \varphi|\delta(\mu^{-1}dT_\lambda\varphi')+\lambda\varepsilon T_\lambda \varphi')_{L^2\Omega^1(M)}\\
&=(\mu^{-1}dT_\lambda \varphi|dT_\lambda\varphi')_{L^2\Omega^2(M)}+(\lambda\varepsilon T_\lambda \varphi|T_\lambda \varphi')_{L^2\Omega^1(M)}\\
&=(\delta(\mu^{-1}dT_\lambda\varphi)+\lambda\varepsilon T_\lambda \varphi|T_\lambda \varphi')_{L^2\Omega^1(M)}\\
&=(\varphi|T_\lambda \varphi')_{L^2\Omega^1(M)}.
\end{align*}
Thus, $T_\lambda^*=T_\lambda$.
\end{proof}

Then $E_0\in H_{d,0}\Omega^1(M)_{0,\varepsilon}$ solves \eqref{eqn::second order equation -- rewritten} if and only if
\begin{equation}\label{eqn::second order equation -- rewritten in terms of T_lambda operators}
E_0-(\omega^2+\lambda)\widetilde T_{\lambda}E_0=T_\lambda\left(\varepsilon J_{e,0}+\delta(\mu^{-1}*J_m)\right)
\end{equation}
where $\widetilde T_{\lambda}=T_\lambda \circ m_\varepsilon\circ P_\varepsilon$, $m_\varepsilon$ is multiplication by $\varepsilon$, and $P_\varepsilon$ is the bounded orthogonal projection of $L^2\Omega^1(M)$ onto $L^2\Omega^1(M)_{0,\varepsilon}$ constructed in Proposition~\ref{prop::Helmholtz decomposition of H_{d,0}}. Note that for all $h'\in H^1_0(M)$ we have
$$
\<\varepsilon J_{e,0}+\delta(\mu^{-1}*J_m),dh'\>_M=(\varepsilon J_{e,0}|dh')_{L^2\Omega^1(M)}+(\mu^{-1}*J_m|d(dh'))_{L^2\Omega^2(M)}=0,
$$
since $J_{e,0}\in L^2\Omega^1(M)_{0,\varepsilon}$. Therefore, by the second part of Proposition~\ref{prop::resonances}, this implies that $T_\lambda\left(\varepsilon J_{e,0}+\delta(\mu^{-1}*J_m)\right)\in H_{d,0}\Omega^1(M)_{0,\varepsilon}$.\smallskip

Second part of Proposition~\ref{prop::resonances} implies also that $\widetilde T_{\lambda}$ can be considered as a bounded linear operator
$$
\widetilde T_{\lambda}:L^2\Omega^1(M)_{0,\varepsilon}\overset{m_\varepsilon}{\longrightarrow} L^2\Omega^1(M)_{0,1}\overset{T_\lambda}{\longrightarrow} H_{d,0}\Omega^1(M)_{0,\varepsilon}\inclusion L^2\Omega^1(M)\overset{P_\varepsilon}{\longrightarrow} L^2\Omega^1(M)_{0,\varepsilon}
$$
and
\begin{equation}\label{eqn::tildeT_lambda maps H_d,0,o,varepsilon into itself}
\widetilde T_{\lambda}:L^2\Omega^1(M)_{0,\varepsilon}\overset{m_\varepsilon}{\longrightarrow} L^2\Omega^1(M)_{0,1}\overset{T_\lambda}{\longrightarrow} H_{d,0}\Omega^1(M)_{0,\varepsilon}.
\end{equation}

The equation \eqref{eqn::second order equation -- rewritten in terms of T_lambda operators} has a unique solution $E_0$ if and only if either $\omega^2=-\lambda$ or $(\omega^2+\lambda)^{-1}\notin\Spec(\widetilde T_{\lambda})$. By Proposition~\ref{prop::compact embedding of H_{d,0,varepsilon} into L^2}, the inclusion $H_{d,0}\Omega^1(M)_{0,\varepsilon}\inclusion L^2\Omega^1(M)$ is compact. This implies that $\widetilde T_{\lambda}$ is compact as an operator from $L^2\Omega^1(M)_{0,\varepsilon}$ to itself. According to Fredholm's alternative (see e.g. \cite[Theorem~0.38]{Folland1995}), this implies that $0\notin \Spec(\widetilde T_{\lambda})$ and $\Spec(\widetilde T_{\lambda})$ is discrete. Therefore, \eqref{eqn::second order equation -- rewritten in terms of T_lambda operators} has a unique solution $E_0$ for any $\omega\notin \Sigma$, where
$$
\Sigma=\{\omega\in \C\setminus\{\pm i\lambda^{1/2}\}:(\omega^2+\lambda)^{-1}\in \Spec(\widetilde T_{\lambda})\}
$$
which is discrete. Since $\id-(\omega^2+\lambda)\widetilde T_{\lambda}:H_{d,0}\Omega^1(M)_{0,\varepsilon}\to H_{d,0}\Omega^1(M)_{0,\varepsilon}$, for all $\omega\notin \Sigma$ we have $(\id-(\lambda+\omega^2)\widetilde T_{\lambda})^{-1}:H_{d,0}\Omega^1(M)_{0,\varepsilon}\to H_{d,0}\Omega^1(M)_{0,\varepsilon}$. Since the right hand-side of \eqref{eqn::second order equation -- rewritten in terms of T_lambda operators} is in $H_{d,0}\Omega^1(M)_{0,\varepsilon}$, this implies that the solution $E_0$ belongs to $H_{d,0}\Omega^1(M)_{0,\varepsilon}$ and
$$
\|E_0\|_{H_d\Omega^1(M)}\le C(\|J_e\|_{L^2\Omega^1(M)}+\|J_m\|_{L^2\Omega^1(M)}),
$$
since $\|\delta(\mu^{-1}*J_m)\|_{(H_d\Omega^1(M))'}\le C\|J_m\|_{L^2(M)}$.\smallskip

Finally, setting $E=E_0-\omega^{-2}dj_e$, we obtain a unique $H_{d,0}\Omega^1(M)$ solution for \eqref{eqn::second order equation} such that
$$
\|E\|_{H_{d}\Omega^1(M)}\le C(\|J_e\|_{L^2\Omega^1(M)}+\|J_m\|_{L^2\Omega^1(M)}),
$$
since $\|j_e\|_{H^1(M)}\le C\|J_e\|_{L^2\Omega^1(M)}$ by \eqref{eqn::estimate for j_e in H^1 norm}. The proof of Theorem~\ref{thm::well posedness new version} is thus complete.

\section{Proof of Theorem~\ref{thm::well posedness new version homogeneous}}\label{section::proof of thm1}
For a fixed $\omega\in \C$, 
consider the following space
$$
\mathcal M_{\varepsilon,\mu,\omega}=\{(E,H)\in H_d\Omega^1(M)\times H_d\Omega^1(M):(E,H)\text{ is a solution of \eqref{eqn::Maxwell homogenous in appendix}}\}.
$$
The topology on this space is the subspace topology in $H_d\Omega^1(M)\times H_d\Omega^1(M)$. It is not difficult to check that $\mathcal M_{\varepsilon,\mu,\omega}$ is closed in $H_d\Omega^1(M)\times H_d\Omega^1(M)$.\smallskip

For a given $(E,H)\in\mathcal M_{\varepsilon,\mu,\omega}$ define $\mathbf{t}_E(E,H):=\mathbf{t}(E)\in TH_d\Omega^1(\p M)$. Since the inclusion $\mathcal M_{\varepsilon,\mu,\omega}\inclusion H_d\Omega^1(M)\times H_d\Omega^1(M)$ is bounded, it is clear that $\mathbf{t}_E:\mathcal M_{\varepsilon,\mu,\omega}\to TH_d\Omega^1(\p M)$ is bounded.\smallskip

We now prove the following proposition which clearly implies Theorem~\ref{thm::well posedness new version homogeneous}.
\begin{Proposition}
There is a discrete set $\Sigma\subset\C$ such that for all $\omega\notin \Sigma$ the operator $\mathbf{t}_E:\mathcal M_{\varepsilon,\mu,\omega}\to TH_d\Omega^1(\p M)$ is a homeomorphism.
\end{Proposition}
\begin{proof}
Let $\Sigma$ be as in Theorem~\ref{thm::well posedness new version} and let us take any $\omega\notin \Sigma$. If we show that the bounded operator $\mathbf{t}_E:\mathcal M_{\varepsilon,\mu,\omega}\to TH_d\Omega^1(\p M)$ is one-to-one and onto, the result follows from Open Mapping Theorem.\smallskip

First, we prove injectivity of $\mathbf{t}_E$. Suppose that $(E_1,H_1), (E_2,H_2)\in \mathcal M_{\varepsilon,\mu,\omega}$ satisfy $\mathbf{t}_E(E_1,H_1)=\mathbf{t}_E(E_2,H_2)$. Then $(E,H)\in \mathcal M_{\varepsilon,\mu}$ and $\mathbf{t}(E)=0$, where $E:=E_1-E_2$ and $H:=H_1-H_2$. Uniqueness part of Theorem~\ref{thm::well posedness new version} (with $J_e=J_m=0$) gives that $E=0$ and $H=0$.\smallskip

Now, we prove surjectivity of $\mathbf{t}_E$. For a given $f\in TH_d\Omega^1(\p M)$, by definition of $TH_d\Omega^1(\p M)$, there is $E'\in H_d\Omega^1(M)$ such that $\mathbf{t}(E')=f$. Applying Theorem~\ref{thm::well posedness new version} with $J_e=i\omega\varepsilon E'$ and $J_m=*dE'$, we obtain a unique $(E_0,H_0)\in H_{d,0}\Omega^1(M)\times H_d\Omega^1(M)$ solving
$$
\begin{cases}
* dE_0=i\omega\mu H_0+*dE',\\
* dH_0=-i\omega \varepsilon E_0+i\omega\varepsilon E'.
\end{cases}
$$
Then $(E,H)\in \mathcal M_{\varepsilon,\mu}$ with $\mathbf{t}_E(E,H)=\mathbf{t}(E)=f$, where $E:=E_0+E'$ and $H:=H_0$.~The proof is complete.
\end{proof}

\section{Proof of Theorem~\ref{thm::spectral theorem for Maxwell equation homogeneous}}\label{section::spectral problem}
For the proof, observe that the boundary value problem \eqref{eqn::Maxwell homogeneous} can be written as
$$
\delta(\mu^{-1}dE)-\omega^2\varepsilon E=0,\quad \mathbf t(E)=0.
$$
We first consider the case $\omega\neq 0$. Then the latter boundary value problem has a solution $E$ in $H_{d,0}\Omega^1(M)_{0,\varepsilon}$ if and only if
$$
E-(\omega^2+\lambda)\widetilde T_{\lambda}E=0,
$$
where $\widetilde T_{\lambda}$ is defined as in the proof of Theorem~\ref{thm::well posedness new version}. We first show that this operator is in fact a self-adjoint operator with respect to certain inner product.

\begin{Lemma}\label{lemma::T_lambda,varepsilon is self-adjoint for real epsilon and mu}
If both $\varepsilon$ and $\mu$ are strictly positive on $M$, then the restriction of $\widetilde T_{\lambda}$ onto $L^2\Omega^1(M)_{0,\varepsilon}$ is self-adjoint with respect to the inner product $(\cdot,\cdot)_{L^2_\varepsilon\Omega^1(M)}$.
\end{Lemma}
\begin{proof}
For $\varphi,\varphi'\in L^2\Omega^1(M)_{0,\varepsilon}$ we have $\widetilde T_{\lambda}\varphi=T_{\lambda}(\varepsilon \varphi)$ and $\widetilde T_{\lambda}\varphi'=T_{\lambda}(\varepsilon \varphi')$. Therefore, using integration by parts,
$$
(\widetilde T_{\lambda}\varphi,\varphi')_{L^2_{\varepsilon}\Omega^1(M)}=(\varepsilon T_\lambda(\varepsilon\varphi)|\varphi')_{L^2\Omega^1(M)}=(T_\lambda(\varepsilon\varphi)|\varepsilon\varphi')_{L^2\Omega^1(M)}.
$$
According to the hypotheses and Proposition~\ref{prop::resonances}, $T_\lambda$ is self-adjoint with respect to the $L^2\Omega^1(M)$-inner product. Therefore,
$$
(\widetilde T_{\lambda}\varphi,\varphi')_{L^2_\varepsilon\Omega^1(M)}=(\varepsilon\varphi|T_\lambda(\varepsilon\varphi'))_{L^2\Omega^1(M)}=(\varphi,\widetilde T_\lambda\varphi')_{L^2_\varepsilon\Omega^1(M)}.
$$
This finishes the proof.
\end{proof}

It was shown in the previous section that the operator $\widetilde T_{\lambda}$ is bounded and compact from $L^2\Omega^1(M)_{0,\varepsilon}$ to itself. Moreover, by Lemma~\ref{lemma::T_lambda,varepsilon is self-adjoint for real epsilon and mu}, the assumptions that $\varepsilon$ and $\mu$ are strictly positive imply that the operator $\widetilde T_{\lambda}$ is self-adjoint with respect to the inner product $(\cdot,\cdot)_{L^2_\varepsilon\Omega^1(M)}$. Then by Fredholm's alternative and Spectral theorem (see e.g. Proposition~6.6 in \cite[Appendix~A]{Taylor1999}) there is a sequence $\{\kappa_k\}_{k=1}^\infty\subset\R$ consisting of eigenvalues of finite multiplicity such that $\kappa_k\searrow 0$ as $k\to\infty$. Associated to the eigenvalues $\kappa_k$ we have the eigenfunctions $e_k\in L^2\Omega^1(M)_{0,\varepsilon}$, forming an orthonormal basis in $L^2\Omega^1(M)_{0,\varepsilon}$ with respect to the inner product $(\cdot,\cdot)_{L^2_\varepsilon\Omega^1(M)}$ and satisfying $\widetilde T_{\lambda}e_k=\kappa_k e_k$. Moreover, each $e_k$ is in $H_{d,0}\Omega^1(M)_{0,\varepsilon}$ (by \eqref{eqn::tildeT_lambda maps H_d,0,o,varepsilon into itself}) and solves
$$
e_k-(\omega_k^2+\lambda)\widetilde T_{\lambda}e_k=0
$$
if $\omega_k^2=\kappa_k^{-1}-\lambda$. Then $e_k$ also solves $\delta(\mu^{-1}de_k)-\omega_k^2\varepsilon e_k=0$. Using this and integrating by parts we show that
$$
(\mu^{-1}de_k|de_k)_{L^2\Omega^1(M)}-\omega_k^2(\varepsilon e_k|e_k)_{L^2\Omega^1(M)}=(\delta(\mu^{-1}de_k)-\omega_k^2\varepsilon e_k|e_k)_{L^2\Omega^1(M)}=0.
$$
Since $\varepsilon$ and $\mu$ are assumed to be strictly positive, this implies that
$$
\omega_k^2=\frac{(\mu^{-1}de_k|de_k)_{L^2\Omega^2(M)}}{(\varepsilon e_k|e_k)_{L^2\Omega^1(M)}}>0.
$$
We may choose $\omega_k>0$, and hence $\omega_k=(\kappa_k^{-1}-\lambda)^{1/2}$. Since $\kappa_k\searrow 0$ as $k\to\infty$, we~have $\omega_k\to \infty$ as $k\to\infty$.\smallskip

Next, we define the sequence $\{h_k\}_{k=1}^\infty\subset L^2\Omega^1(M)$ as $*de_k=i\omega_k\mu h_k$. Then, by direct calculations, it is not difficult to see that each $(e_k,h_k)$ satisfy \eqref{eqn::Maxwell homogenous for eigenvalues and eigenfunctions} and hence also $h_k\in H_d\Omega^1(M)$. Moreover, $h_k\in H_{d}\Omega^1(M)_\mu$, since for all $\varphi\in H_{d}(0,\Omega^1(M))$, integrating by parts, we have
\begin{align*}
(h_k,\varphi)_{L^2_\mu\Omega^1(M)}&=(\mu h_k|\varphi)_{L^2\Omega^1(M)}=(i\omega_k)^{-1}(*de_k|\varphi)_{L^2\Omega^1(M)}\\
&=(i\omega_k)^{-1}(de_k|*\varphi)_{L^2\Omega^2(M)}=(i\omega_k)^{-1}(e_k|*d\varphi)_{L^2\Omega^1(M)}=0.
\end{align*}
Further, using \eqref{eqn::Maxwell homogenous for eigenvalues and eigenfunctions} 
\begin{align*}
(h_k,h_l)_{L^2_\mu\Omega^1(M)}&=(\mu h_k|h_l)_{L^2\Omega^1(M)}=(\omega_k\omega_l)^{-1}(\mu^{-1} de_k|de_l)_{L^2\Omega^2(M)}\\
&=(\omega_k\omega_l)^{-1}(\delta(\mu^{-1} de_k)|e_l)_{L^2\Omega^1(M)}=\frac{\omega_k}{\omega_l}(\varepsilon e_k|e_l)_{L^2\Omega^1(M)}.
\end{align*}
Therefore,
$$
(h_k,h_l)_{L^2_\mu\Omega^1(M)}=\frac{\omega_k}{\omega_l}(e_k,e_l)_{L^2_\varepsilon\Omega^1(M)}=\delta_{kl},
$$
i.e. $\{h_k\}_{k=1}^\infty$ forms an orthonormal set with respect to $(\cdot,\cdot)_{L^2_\mu\Omega^1(M)}$.\smallskip

To show that $\{h_k\}_{k=1}^\infty$ is complete in $H_{d}\Omega^1(M)_\mu$, with respect to $(\cdot,\cdot)_{L^2_\mu\Omega^1(M)}$, take $\psi\in H_{d}\Omega^1(M)_\mu$ such that $(h_k,\psi)_{L^2_\mu\Omega^1(M)}=0$ for all $k\ge 1$ integer. Then
\begin{align*}
0&=i\omega_k(\mu h_k|\psi)_{L^2\Omega^1(M)}=(*de_k|\psi)_{L^2\Omega^1(M)}=(e_k|\delta *\psi)_{L^2\Omega^1(M)}=(e_k|*d\psi)_{L^2\Omega^1(M)}.
\end{align*}
Setting $\phi=\varepsilon^{-1}*d\psi\in L^2\Omega^1(M)$, this implies that $(e_k,\phi)_{L^2_\varepsilon\Omega^1(M)}=0$ for all $k\ge 1$ integer. Suppose that $\phi\in L^2\Omega^1(M)_{0,\varepsilon}$. Then by completeness of $\{e_k\}_{k=1}^\infty$ in $L^2\Omega^1(M)_{0,\varepsilon}$ with respect to the inner product $(\cdot,\cdot)_{L^2_\varepsilon\Omega^1(M)}$, we get $\phi=0$ and hence $\psi\in H_d(0,\Omega^1(M))$. Then $\psi=0$ according to the Helmholtz decomposition \eqref{eqn::decomposition of H_d}.\smallskip

Now, we show that $\phi\in L^2\Omega^1(M)_{0,\varepsilon}$. For this, we need to show that $(\varepsilon \phi|d\varphi)_{L^2\Omega^1(M)}=0$ for all $\varphi\in H_0^1(M)$. By density, it is enough to consider the case when $\varphi\in C^\infty_0\Omega^1(M^{\rm int})$. Then, integrating by parts,
$$
(\varepsilon \phi|d\varphi)_{L^2\Omega^1(M)}=(*d\psi|d\varphi)_{L^2\Omega^1(M)}=(d\psi|*d\varphi)_{L^2\Omega^2(M)}=(\mathbf t(\psi)|\mathbf t(i_\nu*d\varphi))_{\p M}.
$$
Since $\mathbf t(d\varphi)=d_{\p M}(\varphi|_{\p M})=0$, by Lemma~\ref{lemma::expression for boundary term of Stokes' theorem}, we obtain
$$
(\mathbf t(u)|\mathbf t(i_\nu*d\varphi))_{L^2\Omega^1(\p M)}=\int_{\p M}\mathbf t(u)\wedge \mathbf t(d\overline\varphi)=0
$$
for all $u\in C^\infty\Omega^1(M)$. Therefore, $\mathbf t(i_\nu*d\varphi)=0$ and hence $(\varepsilon \phi|d\varphi)_{L^2\Omega^1(M)}=0$. This proves the completeness.\smallskip

Finally, we mention that $\omega=0$ is also an eigenvalue of \eqref{eqn::Maxwell homogeneous} with infinite dimensional eigenspace $H_{d,0}(0,\Omega^1(M))\times H_{d}(0,\Omega^1(M))$.


\begin{thebibliography}{ABC}

\bibitem{CaorsiFernandesRaffetto2000} S. Caorsi, P. Fernandes, M. Raffetto, \emph{On the convergence of Galerkin finite element approximations of electromagnetic eigenproblems}, SIAM J. Numer. Anal., {\bf 38} (2000), 580--607.

\bibitem{Costabel1990}  M. Costabel, \emph{A remark on the regularity of solutions of Maxwell's equations in Lipschitz domains}, Math. Meth. Appl. Sci., {\bf 12} (1990), 365--368.

\bibitem{Costabel1991} M. Costabel, \emph{A coercive bilinear form for Maxwell's equations}, J. Math. Anal. Appl. {\bf 157} (1991), 527--541.

\bibitem{Folland1995} G. B. Folland, \emph{Introduction to Partial Differential Equations}, 2nd edition, Princeton University Press, 1995.

\bibitem{KenigSaloUhlmann2011a} C. Kenig, M. Salo, G. Uhlmann, \emph{Inverse problems for the anisotropic Maxwell equations}, Duke Math. J. {\bf 157} (2011), no. 2, 369--419.

\bibitem{KirschHettlich2015} A. Kirsch, F. Hettlich,  \emph{The Mathematical Theory of Time-Harmonic Maxwell's Equations}, Applied Mathematical Sciences 190, Springer, 2015

\bibitem{Leis1968} R. Leis, \emph{Zur theorie elektromagnetischer schwingungen in anisotropen inhomogenen medien}, Math. Z., {\bf 106} (1968), pp. 213--224.

\bibitem{MitreaMitrea2002} D. Mitrea, M. Mitrea, \emph{Finite Energy Solutions of Maxwell's Equations and Constructive Hodge Decompositions on Nonsmooth Riemannian Manifolds}, J. Func. Anal. {\bf 190} (2002), 339--417.

\bibitem{Mitrea2004} M. Mitrea, \emph{Sharp Hodge decompositions, Maxwell's equations, and vector Poisson problems on nonsmooth, three-dimensional Riemannian manifolds}, Duke Math. J. {\bf 125} (2004), 467--547.

\bibitem{Mitrea2001} M. Mitrea, \emph{Generalized Dirac Operators on Nonsmooth Manifolds and Maxwel's Equations}, J. Fourier Anal. Appl. {\bf 7} (2001), 207--256.

\bibitem{Monk2003} P. Monk, \emph{Finite Element Methods for Maxwell's Equations}, Numer. Math. Sci. Comput., Oxford University Press, New York, 2003.

\bibitem{Schwarz1995} G. Schwarz, \emph{Hodge decomposition -- a method for solving boundary value problems}, Lecture notes in mathematics 1607, Springer, 1995.

\bibitem{Taylor1999} M. E. Taylor, \emph{Partial differential equations I: Basic theory}, Springer, 1999.

\bibitem{Weber1980} C. Weber,  \emph{A local compactness theorem for Maxwell's equations}, Math. Meth. Appl. Sci., {\bf 2} (1980), 12--25.

\end{thebibliography}
\end{document}